\documentclass[12pt,a4paper]{amsart}   

\usepackage[T1]{fontenc}
\usepackage[cp1252]{inputenc}            
\usepackage[dvips]{graphicx}
\usepackage{amssymb}

\newtheorem{theorem}{Theorem}[section]

{\par\vspace{2.5mm}\noindent\textit{Proof.\ }}
{\qed\par\vspace{2.5mm}}

\numberwithin{equation}{section}
\numberwithin{lause}{section}

\allowdisplaybreaks
\begin{document}
\thispagestyle{empty}


\title{Uniqueness of the sum of points of the period five cycle of quadratic polynomials}


\author{Pekka Kosunen}

\date{\today}



\maketitle

\noindent
\address{Department of Physics and Mathematics, University of Eastern Finland, P. O. Box 111, FI-80101 Joensuu, Finland;}
\email{pekka.kosunen@uef.fi}

\begin{abstract}
It is well known that the sum of points of the period five cycle of the quadratic polynomial $f_{c}(x)=x^{2}+c$ is generally not one-valued. In this paper we will show that the sum of cycle points of the curves of period five is at most three-valued on a new coordinate plane, and that this result is essentially the best possible. The method of our proof relies on a implementing Gr\"obner-bases and especially extension theory from the theory of polynomial algebra.
\end{abstract}

\section{Introduction}
The dynamics of quadratic polynomials is commonly studied by using the family of maps  $f_{c}(x)=x^{2}+c$, where $c\in \mathbb{C}$ and $x_{i+1}=f_{c}(x_{i})=x_{i}^{2}+c$. In the article \cite{5} we presented the corresponding iterating system on a new coordinate plane using the change of variables
\begin{equation}
\left\{%
\begin{array}{l}
 u=x+y=x_{0}+x_{1}  \\
v=x+y^2+y-x^2 =x_{0}+x_{1}^{2}+x_{1}-x_{0}^{2}  \label{3}
\end{array}%
\right.
\end{equation}
to the $(x,y)$-plane model (see \cite{1}). In this new $(u,v)$-plane model, equations of periodic curves are of remarkably lower degree than in earlier models.
Now the dynamics of the $(u,v)$-plane is determined by the iteration of the function
\begin{eqnarray*}
G(u,v)&=&(R(u,v),Q(u,v)) \\
&=&\left(\frac{-u+v+uv}{u},\frac{u^2-u+v-u^2v-uv+uv^2+v^2}{u}\right),
\end{eqnarray*}
which is a two-dimensional quadratic polynomial map defined in the complex $2$-space~$\mathbb{C}^{2}$.
The new iteration system is defined recursively as follows:
\begin{equation}
\left\{%
\begin{array}{l}
      (R_{0}(u,v),Q_{0}(u,v))=(u,v)=(u_{0},v_{0}), \\ \\
      (R_{1}(u,v),Q_{1}(u,v))=(R(u,v),Q(u,v))=(u_{1},v_{1}), \\ \\
      (R_{n+1}(u,v),Q_{n+1}(u,v))=G(R_{n}(u,v),Q_{n}(u,v))=(u_{n+1},v_{n+1}),\label{5}
\end{array}%
\right.
\end{equation}
where
\begin{equation}
\left\{%
\begin{array}{l}
      R_{n+1}(u,v)=Q_{n}(u,v)-1+Q_{n}(u,v)/R_{n}(u,v) \\ \\
      Q_{n+1}(u,v)=R_{n+1}(u,v)(1+Q_{n}(u,v)-R_{n}(u,v)),\label{6}
\end{array}%
\right.
\end{equation}
and $n\in \mathbb{N}\cup\{0\}$. Now $(u,v)$ is fixed $G^{n}$, so $G^{n}(u,v)=(u,v)$, if and only if $(R_{n}(u,v),Q_{n}(u,v))=(u,v)$. The set of such points is the union of all orbits, whose period divides $n$, and the set of periodic points of period $n$ are the points with exact period dividing $n$.

In complex dynamics, the sum of period cycle points has been a commonly used parameter in many connections (see, for example, \cite{4}, \cite{13}, \cite{1}, \cite{17} and \cite{2}). In the article \cite{17} Giarrusso and Fisher used it to the parameterization of the period $3$ hyperbolic components of the Mandelbrot set. Later, in the article \cite{1}, Erkama studied the case of the period $3-4$ hyperbolic components of the Mandelbrot set on the $(x,y)$-plane and completely solved both cases.

Moreover, Erkama \cite{1} has shown that the sum of periodic orbit points
\begin{displaymath}
S_{n}=x_{0}+x_{1}+x_{2}+ \ldots +x_{n-2}+x_{n-1}
\end{displaymath}
is unique when $n=3$ or $n=4$. Conversely, the sum of cyclic points of period three and four determines these orbits uniquely. In the period five case this situation changes and the sum of the cycle points is no longer unique. We can see this property in the articles \cite{4} and \cite{2}, in which Brown and Morton have formed the so called \textit{trace} formulas in the cases of period five and six using $c$ and the sum of period cycle points as parameters. In this paper we will show that by implementing the change of variables ($\ref{3}$), we obtain a new coordinate plane where the sum of period five cycle points is at most three-valued, and show that no better result is obtainable in this coordinate plane. This is done by applying methods of polynomial algebra (without the classical trace formula), as our proof relies on the use of the elimination theory and especially the extension theorem \cite{12}. The extension theorem tells us the best possible result (which the trace formula does not necessarily do) due to the use of Gr\"obner-basis. In the next section we present the most central tools and constructions related to these theorems.

\section{A brief introduction to the elimination theorem}


We start with the \textit{Hilbert basis theorem}:
Every \textit{ideal} $I \subset \mathbb{C}[x_{1},
\ldots ,x_{n}]$ has a finite generating set. That is
$I=\langle g_{1}, \ldots , g_{t} \rangle$
for some $g_{1}, \ldots , g_{t} \in I$.
Hence $\langle g_{1}, \ldots , g_{t} \rangle$ is the ideal generated by the elements $ g_{1}, \ldots , g_{t}$,
in other words $ g_{1}, \ldots , g_{t}$ is the basis of the ideal. The so called \textit{Gr\"obner-basis} has proved to be especially useful in many connections \cite{12}, for example in kinematic analysis of mechanisms (see \cite{16} and \cite{15}). In order to introduce this basis we need the following constructions.

Let the $f\in \mathbb{C}[x_{1}, \ldots ,x_{n}]$ be the polynomial given by
\begin{displaymath}
f=\sum_{\alpha}a_{\alpha}x^{\alpha},
\end{displaymath}
where $a_{\alpha}\in \mathbb{C}$, $\alpha=(\alpha_{1},\ldots ,\alpha_{n})$ and $x^{\alpha}=x_{1}^{\alpha_{1}} \cdot x_{2}^{\alpha_{2}} \cdots  x_{n}^{\alpha_{n}}$ is a monomial. Then the \textit{multidegree} of $f$ is
\begin{displaymath}
\displaystyle{mdeg}(f)=max \{\alpha \mid a_{\alpha}\neq 0\},
\end{displaymath}
the \textit{leading coefficent} of $f$ is
\begin{displaymath}
LC(f)=a_{mdeg(f)},
\end{displaymath}
the \textit{leading monomial} of $f$ is
\begin{displaymath}
LM(f)=x^{mdeg(f)},
\end{displaymath}
and the \textit{leading term} of $f$ is
\begin{displaymath}
LT(f)=LC(f)LM(f).
\end{displaymath}

To calculate a Gr\"obner basis of an ideal we need to order terms of polynomials by using a \textit{monomial ordering}. A Gr\"obner basis can be calculated by using any monomial ordering, but differences in the number of operations can be very significant. An effective tool to calculate the Gr\"obner basis is the software \textit{Singular}, which has been especially designed for operating with polynomial equations. Next we will define a monomial ordering of nonlinear polynomials.

Relation $<$ is the \textit{linear ordering} in the set $S$, if $x<y$, $x=y$ or $y<x$ for all $x,y \in S$.
A monomial ordering in the set $\mathbb{N}^{n}$  is a relation $\prec$ if
\begin{displaymath}
\begin{array}{lll}
1. && \prec \ \textrm{is} \ \textrm{linear} \ \textrm{ordering},  \\
2. && \textrm{implication} \ x^{\alpha} \prec x^{\beta}  \Rightarrow x^{\alpha+\gamma} \prec x^{\beta +\gamma} \ \textrm{holds} \ \textrm{for} \ \textrm{all} \ \alpha,\beta,\gamma \in \mathbb{N}^{n},  \\
3. && x^{\alpha}>1.
\end{array}
\end{displaymath}
To compute elimination ideals we need \textit{product orderings}. Let $\succ_{A}$ be an ordering for the variable $x$, and let $\succ_{B}$ be ordering for the variable $y$ in the ring \\ $\mathbb{C}[x_{1},
\ldots ,x_{n},y_{1}, \ldots ,y_{m}]$. Now we can define the product ordering as follows:
\begin{displaymath}
x^{\alpha}y^{\beta}\succ x^{\gamma}y^{\delta} \qquad \textrm{if} \qquad
\left\{%
\begin{array}{ll}
      x^{\alpha} \succ_{A}  x^{\gamma} \quad & \textrm{or} \\
       x^{\alpha} =  x^{\gamma} \quad &\textrm{and} \quad y^{\beta} \succ_{B} y^{\delta}.
\end{array}%
\right.
\end{displaymath}
There are several monomial orders but we need only the \textit{lexicographic order} $\prec_{lex}$ in the elimination theory.
Let $\alpha,\beta\in \mathbb{N}^{n}$. Then we say that $x^{\alpha} \prec_{lex} x^{\beta}$ if  $x^{\alpha_{1}}=x^{\beta_{1}},  \ldots, x^{\alpha_{k-1}}=x^{\beta_{k-1}}, x^{\alpha_{k}} \succeq x^{\beta_{k}}$
and $x^{\alpha_{k+1}} \prec x^{\beta_{k+1}}$.
One of the most important tools in the elimination theory is the \textit{Gr\"obner basis of an ideal}:
Fix a monomial order. A finite subset
\begin{displaymath}
G_{I}=\{g_{1}, \ldots , g_{t}\} \subset I
\end{displaymath}
of an ideal $I$ is said to be a Gr\"obner basis (or standard basis) if
\begin{displaymath}
\langle LT(g_{1}) , \ldots , LT(g_{t}) \rangle=\langle LT(I) \rangle.
\end{displaymath}
Based on the Hilbert basis theorem we know that every ideal $I \subset \mathbb{C}[x_{1},
\ldots ,x_{n}]$ has a Gr\"obner basis $G_{I}=\{g_{1}, \ldots , g_{s}\}$ so that
\begin{displaymath}
\langle G_{I} \rangle=I.
\end{displaymath}

It is essential to construct also an affine variety corresponding to the ideal.
Let $f_{1}, \ldots , f_{s}$ be polynomials in the ring $\mathbb{C}[x_{1},
\ldots ,x_{n}]$. Then we set
\begin{displaymath}
\mathbf{V}(f_{1}, \ldots , f_{s})=\{(a_{1} , \ldots ,a_{n}) \in
\mathbb{C}^{n}:f_{i}(a_{1} , \ldots ,a_{n})=0 \ \textrm{for} \ \textrm{all} \ 1\leq i \leq s \},
\end{displaymath}
and we call $\mathbf{V}(f_{1}, \ldots , f_{s})$ as the \textit{affine variety} defined by $f_{1}, \ldots ,f_{s}$.
Now if $I=\langle f_{1}, \ldots , f_{s} \rangle$, so $\mathbf{V}(I)=\mathbf{V}(f_{1}, \ldots , f_{s})$
and naturally we obtain the variety of the ideal as the variety of its Gr\"obner basis:
 $\mathbf{V}(I)=\mathbf{V}(\langle G_{I}  \rangle)$.

 When we consider ideals and their algebraic varieties we are sometimes just interested about polynomials $f\in \mathbb{C}[x_{1},
\ldots ,x_{n}]$, which belong to the original ideal $f\in I$, but contain only certain variables of the ring variables of $\mathbb{C}[x_{1},
\ldots ,x_{n}]$. For this purpose we need \textit{elimination ideals}.
Let $I=\langle f_{1}, \ldots ,f_{s} \rangle \subset \mathbb{C}[x_{1},
\ldots ,x_{n}]$. The $k$:th elimination ideal $I_{k}$ is the ideal of $\mathbb{C}[x_{k+1}, \ldots ,x_{n}]$ defined by
\begin{displaymath}
I_{k}=I \cap \mathbb{C}[x_{k+1}, \ldots ,x_{n}].
\end{displaymath}
Next we give the an important \textit{elimination theorem} which we use in our proof.
\begin{theorem}[The Elimination Theorem]\label{elim}
Let $I \subset \mathbb{C}[x_{1},
\ldots ,x_{n}]$ be an ideal and let $G$ be a Gr\"obner basis of $I$ with respect to lexicographic order, where $x_{1}
\succ x_{2} \succ \ldots \succ x_{n}$. Then for every $0\leq k \leq n$, the set
\begin{displaymath}
G_{k}=G \cap \mathbb{C}[x_{k+1}, \ldots ,x_{n}]
\end{displaymath}
is a Gröbner basis of the $k$:th elimination ideal $I_{k}$.
\end{theorem}
The elimination theorem is closely related to the \textit{extension theorem}, which tells us the correspondence between varieties of the original ideal and the elimination ideal. In other words, if we apply this theorem to a system of equations we see whether the partial solution $V(I_{k})$ of the system of equations is also a solution of the whole system $V(I)$.


\begin{theorem}[The Extension Theorem]  \label{5a}
Let $I=\langle f_{1}, \ldots ,f_{s} \rangle \subset
\mathbb{C}[x_{1}, \ldots ,x_{n}]$ and let $I_{1}$ be the first elimination ideal of $I$.
For each $1\leq i \leq s$ write $f_{i}$ in the form
\begin{displaymath}
f_{i}=g_{i}(x_{2}, \ldots , x_{n})x_{1}^{N_{i}}+ \ \textrm{terms} \ \textrm{in} \ \textrm{which} \ \textrm{deg}(x_{1}) < N_{i},
\end{displaymath}
where $N_{i}\geq 0$ and $g_{i} \subset \mathbb{C}[x_{2}, \ldots
,x_{n}]$ , $g_{i}\neq 0$.
Suppose that we have a partial solution $(a_{2}, \ldots , a_{n})\in
\mathbf{V}(I_{1})$. If $(a_{2}, \ldots , a_{n}) \notin
\mathbf{V}(g_{1}, \ldots ,g_{s})$, then there exists $a_{1}\in
\mathbb{C}$ such that $(a_{1},a_{2}, \ldots , a_{n}) \in \mathbf{V}(I)$.

\end{theorem}

\section{On properties of points sums of period $3-5$ cycles}
In this section we first prove the uniqueness properties of points sums of cycles of period three and four by using methods from polynomial algebra in a new way. After this we concentrate on the period five case and show that the sum of period five cycle points is at most three-valued. The next result shows the relation between the sums of cycle points of the $(x,y)$-plane \cite{1} and the $(u,v)$-plane \cite{5}:
\begin{theorem} \label{a}
Let $x_{0}, x_{1}, x_{2}, \ldots, x_{n-1}$ periodic $n$ orbit points. If
\begin{displaymath}
S_{n}=x_{0}+x_{1}+x_{2}+ \ldots +x_{n-2}+x_{n-1},
\end{displaymath}
 then by transformation $(\ref{3})$ and $(\ref{5})$
\begin{displaymath}
S_{n}=\frac{1}{2}S_{n}^{1}=\frac{1}{2}S_{n}^{2},
\end{displaymath}
where
\begin{displaymath}
S_{n}^{1}=u_{0}+u_{1}+u_{2}+ \ldots +u_{n-2}+u_{n-1}
\end{displaymath}
and
\begin{displaymath}
S_{n}^{2}=v_{0}+v_{1}+v_{2}+ \ldots +v_{n-2}+v_{n-1}.
\end{displaymath}
\end{theorem}
\begin{proof}
By writing out both components we obtain
\begin{displaymath}
\begin{array}{lll}
S_{n}^{1}&=&u_{0}+u_{1}+u_{2}+ \ldots +u_{n-2}+u_{n-1}\\
&=&x_{0}+x_{1}+x_{1}+x_{2}+x_{2}+x_{3}+ \ldots +x_{n-2}+x_{n-1}+x_{n-1}+x_{n} \\
&=&x_{0}+x_{1}+x_{1}+x_{2}+x_{2}+x_{3}+ \ldots +x_{n-2}+x_{n-1}+x_{n-1}+x_{0} \\
&=&2(x_{0}+x_{1}+x_{2}+x_{3}+ \ldots +x_{n-2}+x_{n-1})
\end{array}
\end{displaymath}
and similarly
\begin{displaymath}
\begin{array}{lll}
S_{n}^{2}&=&v_{0}+v_{1}+v_{2}+ \ldots +v_{n-2}+v_{n-1}\\
&=&x_{0}+x_{2}+x_{1}+x_{3}+x_{2}+x_{4}+ \ldots +x_{n-2}+x_{n}+x_{n-1}+x_{n+1}\\
&=&2(x_{0}+x_{1}+x_{2}+ \ldots x_{n-2}+x_{n-1}).
\end{array}
\end{displaymath}
\end{proof}

\subsection{The uniqueness of cycle points sums of period three and four orbits}
The sums of points of the period three and four cycles are obtained in \cite{1} as
\begin{equation}
S_{3}=x_{0}+x_{1}+x_{2} \label{cc1}
\end{equation}
and
\begin{equation}
S_{4}=x_{0}+x_{1}+x_{2}+x_{3}. \label{cc2}
\end{equation}
According to Theorem $\ref{a}$ and
by using the formula ($\ref{5}$) we obtain on the $(u,v)$-plane
\begin{equation}
S_{3}(u,v)=\frac{1}{2}(u_{0}+u_{1}+u_{2})= \,{\frac {{u}^{2}-u+v+2\,uv}{2u}} \label{2aa1}
\end{equation}
and
\begin{equation}
S_{4}(u,v)=\frac{1}{2}(u_{0}+u_{1}+u_{2}+u_{3})= -{\frac {-v+u-{u}^{2}+{u}^{2}v-u{v}^{2}-{v}^{2}}{u}}. \label{2aa2}
\end{equation}
Based on the article \cite{5}, the equations of periodic orbits of period three and four are $P_{3}(u,v)=0$ and $P_{4}(u,v)=0$, where
\begin{displaymath}
P_{3}(u,v)=uv+1+v,
\end{displaymath}
and
\begin{eqnarray}
P_{4}(u,v)&=&-{u}^{2}v+{u}^{2}{v}^{2}-u+uv+u{v}^{2}-{v}^{2}-{v}^{3}-u{v}^{3} \nonumber \\
&=&u^{2}(-v^{2}+v)+u(v^{3}-v^{2}-v+1)+v^{3}+v^{2}.\nonumber
\end{eqnarray}
Now we form polynomials $B_{3}(u,v,S_{3})=0$ and $B_{4}(u,v,S_{4})=0$ based on formulas ($\ref{2aa1}$) and ($\ref{2aa2}$) as
\begin{eqnarray}
B_{3}(u,v,S_{3})&=&2uS_{3}-({u}^{2}-u+v+2\,uv) \nonumber \\
&=&-{u}^{2}+ \left( 2\,S_{3}+1-2\,v \right) u-v \nonumber
\end{eqnarray}
and
\begin{eqnarray}
B_{4}(u,v,S_{4})&=&uS_{4}-v+u-{u}^{2}+{u}^{2}v-u{v}^{2}-{v}^{2} \nonumber \\
&=&\left( -1+v \right) {u}^{2}+ \left( S_{4}-{v}^{2}+1 \right) u-v-{v}^{2}.\nonumber
\end{eqnarray}
Based on the previous equations we can form the pair of equations
\begin{equation}
\left\{%
\begin{array}{lll}
   P_{3}(u,v)&=& 0  \\
   B_{3}(u,v,S_{3})&=& 0  \label{222a}
    \end{array}%
\right.
\end{equation}
and
\begin{equation}
\left\{%
\begin{array}{lll}
   P_{4}(u,v)&=& 0  \\
   B_{4}(u,v,S_{4})&=& 0,  \label{222b}
    \end{array}%
\right.
\end{equation}
and obtain the ideals
\begin{eqnarray}
I_{3}&=& \langle P_{3}(u,v),B_{3}(u,v,S_{3})\rangle \nonumber \\
&=&\langle uv+1+v, -{u}^{2}+ \left( 2\,S_{3}+1-2\,v \right) u-v \rangle, \nonumber
\end{eqnarray}
and
\begin{eqnarray}
I_{4}&=&\langle
P_{4}(u,v),B_{4}(u,v,S_{4}) \rangle \nonumber \\ &=& \langle  u^{2}(-v^{2}+v)+u(v^{3}-v^{2}-v+1)+v^{3}+v^{2}, \nonumber \\
&& \left( -1+v \right) {u}^{2}+ \left( S_{4}-{v}^{2}+1 \right) u-v-{v}^{2} \rangle. \nonumber
\end{eqnarray}
We eliminate from these ideals the variable $u$ and obtain the Gr\"obner-basis of the eliminated ideals $I_{3u}$ and $I_{4u}$ to calculate the Gr\"obner-basis of the ideals $I_{3}$ and $I_{4}$ using the Singular program (\cite{10}). Gr\"obner-bases of the ideals $I_{3}$ and $I_{4}$, by using the ordering  $\prec_{lex}$, where $S_{3} \prec_{lex} v \prec_{lex} u$ and $S_{4} \prec_{lex} v \prec_{lex} u$, are
\begin{displaymath}
G_{3}=\{g_{31},g_{32}\},
\end{displaymath}
where
\begin{displaymath}
\begin{array}{lll}
g_{31}&=&v^{3}-2v^{2}S_{3}-2vS_{3}-3v-1, \\
g_{32}&=&u+v^{2}-2vS_{3}-2S_{3}-2,
\end{array}
\end{displaymath}
and
\begin{displaymath}
G_{4}=\{g_{41},g_{42},g_{43},g_{44}\},
\end{displaymath}
where
\begin{displaymath}
\begin{array}{lll}
g_{41}&=& v^{4}-v^{3}S_{4}+v^{3}-v^{2}S_{4}-v^{2}-v, \\
      g_{42} &=& uv^{2}-uvS_{4}-u, \\
       g_{43}&=& u^{2}S_{4}-uv^{3}+uv^{2}S_{4}-uv^{2}+uvS_{4}+uv-uS_{4}^{2}+u-v^{3}+v^{2}S_{4} \\
       &&-2v^{2}+vS_{4}-v, \\
       g_{44}&=& u^{2}v-u^{2}-uv^{2}+uS_{4}+u-v^{2}-v.
\end{array}
\end{displaymath}
Thus $g_{31}$ and $g_{41}$ depend only on the variables $v$ and $S_{5}$. Based on the elimination Theorem $\ref{elim}$ the set
\begin{displaymath}
G_{3u}=G_{3} \cap \mathbb{C}[v,S_{3}]=\{g_{31}\}
\end{displaymath}
is the Gr\"obner-basis of the elimination ideal $I_{3u}$ and so $V(I_{3u})=V(g_{31})$.
At the same way the set
\begin{displaymath}
G_{4u}=G_{4} \cap \mathbb{C}[v,S_{4}]=\{g_{41}\}
\end{displaymath}
is the Gr\"obner-basis of the elimination ideal $I_{4u}$ and so $V(I_{4u})=V(g_{41})$.
In the case $g_{31}=0$ it follows that
\begin{equation}
S_{3}=\,{\frac {-1-3\,v+{v}^{3}}{2v \left( v+1 \right) }}. \label{ab1}
\end{equation}
If $g_{41}=0$
we have
\begin{equation}
S_{4}={\frac {-{v}^{4}-{v}^{3}+{v}^{2}+v}{-{v}^{3}-{v}^{2}}}=\frac{v^{2}-1}{v}. \label{ab2}
\end{equation}
As we can see, in both cases the sum of the points of cycles of the given period is unique. In other words, the orbit sums $S_{3}$ and $S_{4}$ uniquely determine the orbit.
If we eliminate in the first case the variable $v$ instead of the variable $u$, we obtain the Gr\"obner-basis
\begin{displaymath}
G_{3v}=u^{3}-2u^{2}S_{3}-2uS_{3}-3u-1,
\end{displaymath}
which gives the same result as ($\ref{ab1}$). However, the same procedure in the period four case produces the Gr\"obner-basis
\begin{displaymath}
\begin{array}{ll}
G_{4v}&=u^{5}S_{4}-2u^{4}S_{4}^{2}+u^{3}S_{4}^{3}-u^{3}S_{4}^{2}-2u^{3}S_{4}-4u^{3}+u^{2}S_{4}^{3}+2u^{2}S_{4}^{2}+4u^{2}S_{4}+uS_{4} \\
&=\left( {u}^{3}+{u}^{2} \right) {S_{4}}^{3}+ \left( -{u}^{3}+2\,{u}^{2}-2
\,{u}^{4} \right) {S_{4}}^{2}+ \left( {u}^{5}-2\,{u}^{3}+u+4\,{u}^{2}
 \right) S_{4}-4\,{u}^{3}
\end{array}
\end{displaymath}
and this is of higher degree than ($\ref{ab2}$).

\subsection{On the uniqueness of the cycle points sum of period five orbits}
Next we prove that in the case of period five cycles, the sum of period five points is at most three-valued. We use in this proof the Gr\"obner-basis of an ideal, like before in period three and four cases, which produce for us the Gr\"obner-basis of the elimination ideal. Because this method relies on bases, the following result is optimal.
\begin{theorem}\label{per5}
The sum of period five cycle points is at most three-valued.
\end{theorem}

\begin{proof}
By the article \cite{5}, the equation for period five orbit on the $(u,v)$-plane is of the form $P_{5}(u,v)=0$, where
\begin{eqnarray}
P_{5}(u,v)&=&u^{7}(-{v}^{4}+2{v}^{3}-{v}^{2})+u^{6}(3{v}^{5}-8{v}^{4}+5{v}^{3}+{v}^{2}-v) \nonumber \\
       &&+u^{5}(3{v}^{6}+14{v}^{5}-12{v}^{4}-5{v}^{3}+7{v}^{2}-v)+u^{4}({v}^{7}  \nonumber \\
       && -12{v}^{6}+18{v}^{5}+6{v}^{4}-16{v}^{3}+3{v}^{2}+2v)+u^{3}(4{v}^{7}  \label{1123} \\
       && -16{v}^{6}+19{v}^{4}-5{v}^{3}-4{v}^{2}+2v+1)+u^{2}(6{v}^{7}-6{v}^{6} \nonumber \\
       && -12{v}^{5}+6{v}^{4}+4{v}^{3}-2{v}^{2})+u(4{v}^{7}+3{v}^{6}-4{v}^{5} \nonumber \\
       && -2{v}^{4}+{v}^{3})+{v}^{7}+2{v}^{6}+{v}^{5}.\nonumber
\end{eqnarray}
According to the Theorem $\ref{a}$, the sum
\begin{equation}
S_{5}=x_{0}+x_{1}+x_{2}+x_{3}+x_{4} \label{cc}
\end{equation}
of the period five points satisfies
\begin{displaymath}
S_{n}=\frac{1}{2}S_{n}^{1}=\frac{1}{2}S_{n}^{2},
\end{displaymath}
 and based on the formula ($\ref{5}$) we obtain
\begin{eqnarray}
S_{5}(u,v)&=& \frac{1}{2}(u_{0}+u_{1}+u_{2}+u_{3}+u_{4}) \nonumber \\ \label{2aa}
 &=& \displaystyle\frac
{-3\,{u}^{2}+4\,{u}^{2}v+3\,uv-4\,{u}^{3}{v}^{3}+2\,{u}^{2
}{v}^{4}+4\,u{v}^{4}+{u}^{3}-2\,{u}^{4}v-2\,{u}^{3}v}{2{u}^{2}}
\\ \nonumber & &+\frac{+2\,{u}^{2}{v}^{2}
-2\,u{v}^{2}+2\,{v}^{3}+2\,{u}^{4}{v}^{2}+6\,{u}^{3}{v}^{2}-8\,{u}^{2}
{v}^{3}-2\,u{v}^{3}+2\,{v}^{4}}{2{u}^{2}} \nonumber
\end{eqnarray}
on the $(u,v)$-plane.
We form from this the polynomial
\begin{displaymath}
\begin{array}{lll}
B_{5}(u,v,S_{5})&=&u^{4}(2{v}^{2}-2v)+u^{3}(-4{v}^{3}+6{v}^{2}-2v+1)\\
&&+u^{2}(-2S_{5}+2{v}^{4}-8{v}^{3}+2{v}^{2}+4v-3)\\
&&+u(4{v}^{4}-2{v}^{3}-2{v}^{2}+3v)+2{v}^{4}+2{v}^{3}.
\end{array}
\end{displaymath}
Now we can form the pair of equations
\begin{equation}
\left\{%
\begin{array}{lll}
   P_{5}(u,v)&=& 0  \\
   B_{5}(u,v,S_{5})&=& 0,  \label{222}
    \end{array}%
\right.
\end{equation}
and the two polynomials $P_{5}(u,v)$ and $B_{5}(u,v,S_{5})$ form an ideal
\begin{eqnarray}
I_{5}&=&\langle
P_{5}(u,v),B_{5}(u,v,S_{5})\rangle \nonumber \\ &=&\langle a_{7}u^{7}+a_{6}u^{6}+a_{5}u^{5}+a_{4}u^{4}+a_{3}u^{3}+a_{2}u^{2}+a_{1}u+a_{0}, \nonumber \\
&& b_{4}u^{4}+b_{3}u^{3}+b_{2}u^{2}+b_{1}u+b_{0}  \rangle, \nonumber
\end{eqnarray}
where
\begin{eqnarray}
&&a_{0}= {v}^{7}+2{v}^{6}+{v}^{5}\nonumber\\
&&a_{1}=4{v}^{7}+3{v}^{6}-4{v}^{5}-2{v}^{4}+{v}^{3} \nonumber\\
&&a_{2}= 6{v}^{7}-6{v}^{6}-12{v}^{5}+6{v}^{4}+4{v}^{3}-2{v}^{2}\nonumber\\
&&a_{3}=4{v}^{7}-16{v}^{6}+19{v}^{4}-5{v}^{3}-4{v}^{2}+2v+1 \nonumber\\
&&a_{4}={v}^{7}-12{v}^{6}+18{v}^{5}+6{v}^{4}-16{v}^{3}+3{v}^{2}+2v \nonumber\\
&&a_{5}= 3{v}^{6}+14{v}^{5}-12{v}^{4}-5{v}^{3}+7{v}^{2}-v\nonumber\\
&&a_{6}= 3{v}^{5}-8{v}^{4}+5{v}^{3}+{v}^{2}-v\nonumber\\
&&a_{7}= -{v}^{4}+2{v}^{3}-{v}^{2}\nonumber\\
&&b_{0}= 2{v}^{3}\nonumber\\
&&b_{1}= 4{v}^{4}-2{v}^{3}-2{v}^{2}+3v\nonumber\\
&&b_{2}= -2S_{5}+2{v}^{4}-8{v}^{3}+2{v}^{2}+4v-3\nonumber\\
&&b_{3}= -4{v}^{3}+6{v}^{2}-2v+1\nonumber\\
&&b_{4}=2{v}^{2}-2v. \nonumber
\end{eqnarray}
We eliminate from this the variable $u$ by forming the Gr\"obner-basis $G_{5u}$ of the elimination ideal $I_{5u}$ in order to calculate the Gr\"obner-basis $G_{5}$ of the ideal $I_{5}$ using Singular. We obtain the Gr\"obner-basis of the ideal $I$ as
\begin{displaymath}
G_{5}=\{g_{51},g_{52},g_{53},g_{54},g_{55},g_{56}\}
\end{displaymath}
using ordering  $\prec_{lex}$, where $S_{5} \prec_{lex} v \prec_{lex} u$.
Here $g_{51},g_{52},g_{53},g_{54}$ and $g_{55}$ depend on the variables $u$, $v$ and $S_{5}$, and $g_{56}$ depends only on the variables $v$ and $S_{5}$. By the elimination theorem the set
\begin{displaymath}
G_{5u}=G \cap \mathbb{C}[v,S_{5}]=\{g_{56}\}
\end{displaymath}
is the Gr\"obner-basis of the elimination ideal $I_{5u}$ and so $V(I_{5u})=V(g_{56})$.
Now the Gr\"obner-basis of the elimination ideal $I_{5u}$ is of the form
\begin{eqnarray}
&G_{5u}=&v^{6}(v+1)^{2}(c_{0}{v}^{15}+c_{1}{v}^{14}+c_{2}{v}^{13}+c_{3}{v}^{12}+c_{4}{v}^{11}+c_{5}{v}^{10}+c_{6}{v}^{9} \label{G5} \\
      &&+c_{7}{v}^{8}+c_{8}{v}^{7}+ c_{9}{v}^{6}+c_{10}{v}^{5}+c_{11}{v}^{4}+c_{12}{v}^{3}+c_{13}{v}^{2}+c_{14}v+c_{15}), \nonumber
\end{eqnarray}
where

\begin{eqnarray}
 c_{0}&=&27 \nonumber \\
 c_{1}&=&-162S_{5} \nonumber \\
 c_{2}&=& 252S_{5}^{2}-432\,S_{5}-684 \nonumber \\
 c_{3}&=& 280S_{5}^{3}+2592S_{5}^{2}+4128\,S_{5}+556 \nonumber \\
 c_{4}&=& -1264S_{5}^{4}-5760S_{5}^{3}-8712S_{5}^{2}+
236\,S_{5}+4002 \nonumber \\
 c_{5}&=& 1440S_{5}^{5}+5888S_{5}^{4}+6864
S_{5}^{3}-8440S_{5}^{2}-19596\,S_{5}-4336 \nonumber \\
 c_{6}&=&-704S_{5}^{6}-2816S_{5}^{5}+320S_{5}
^{4}-8380+19584
S_{5}^{3}+37536S_{5}^{2}+11528\,S_{5} \nonumber \\
 c_{7}&=& 128S_{5}^{7}+512S_{5}^{6}-3328S_{5}^{5}-18112S_{5}^{4}-30144S_{5}^{3}+
1120S_{5}^{2}\nonumber \\
&&+39192\,S_{5}+14868 \nonumber \\
 c_{8}&=& 1664S_{5}^{6}+7488S_{5}^{5}+7824S_{5}^{4}-
21520S_{5}^{3}-64076S_{5}^{2}-38238\,S_{5} +4003\nonumber \\
 c_{9}&=&-256S_{5}^{7}-1152S_{5}^{6}+1952S_{5}^{5}+19360
S_{5}^{4}+44040S_{5}^{3}+22980S_{5}^{2}\nonumber \\
&&-29970\,S_{5}-
19924 \nonumber \\
 c_{10}&=& -
1216S_{5}^{6}-6336S_{5}^{5}-11216S_{5}^{4}+5848S_{5}^{3}+46108S_{5}^{2}\nonumber \\
&&+43516\,
S_{5}+5736  \nonumber \\
 c_{11}&=& 128S_{5}^{7}+640S_{5}^
{6}-160S_{5}^{5}-8208S_{5}^{4}-25384S_{5}^{3}-25368S_{5}^{2}\nonumber \\
&&+3504\,S_{5} +10380\nonumber \\
 c_{12}&=& 256S_{5}^{6}+1664S_{5}^{5}-16730\,S_{5}+4432S_{5}^{4}+2056S_{5}^{3}-11160
S_{5}^{2}-4909  \nonumber \\
 c_{13}&=&96S_{5}^{5}+1104S_{5
}^{4}+4240S_{5}^{3}+6396S_{5}^{2}+2070\,S_{5} -1934 \nonumber \\
 c_{14}&=&216S_{5}^{3}+1068S_{5}^{2}+1974\,S_{5}+1347  \nonumber \\
 c_{15}&=&-27. \nonumber
\end{eqnarray}

By ($\ref{G5}$) $G_{5u}$ is formed as a product of three terms. We denote the last of these terms in ($\ref{G5}$) by
$C(v,S_{5})$. Now we obtain the variety $V(I_{u})$ of the elimination ideal as the union of three varieties corresponding to the factors of $G_{5u}$ of as follows
\begin{eqnarray}
\mathbf{V}(I_{5u})&=&\mathbf{V}(v^{6}) \bigcup
 \mathbf{V}\left((v+1)^{2}\right) \bigcup
 \mathbf{V}\left(C(v,S_{5})\right) \nonumber \\
&=& \left\{(0,S_{5}),(-1,S_{5})\right\} \bigcup
 \mathbf{V}\left(C(v,S_{5})\right). \nonumber
\end{eqnarray}\\
Note that $G_{5u}$ is of degree $23$ with respect to the variable $v$ and of degree $7$
with respect to the variable $S_{5}$. We denote, according to the extension theorem,
\begin{displaymath}
f_{i}=g_{i}(v,S_{5})u^{N_{i}}+ \ \textrm{terms} \ \textrm{such} \ \textrm{that} \ \textrm{deg}(u) < N_{i},\ i=1,2
\end{displaymath}
where
\begin{displaymath}
g_{1}=a_{7}=-{v}^{4}+2{v}^{3}-{v}^{2}
\end{displaymath}
and
\begin{displaymath}
g_{2}=b_{4}=2{v}^{2}-2v .
\end{displaymath}
The corresponding varieties are
\begin{displaymath}
\mathbf{V}(g_{1})=\{(0,S_{5}),(1,S_{5})\}=\mathbf{V}(g_{2}),
\end{displaymath}
 so
\begin{displaymath}
  \mathbf{V}(g_{1},g_{2})=\mathbf{V}(g_{1}) \bigcap
 \mathbf{V}(g_{2})=\{(0,S_{5}),(1,S_{5})\}.
\end{displaymath}
In other words for all $v\neq 0$ and $v\neq 1$  we have $(v, S_{5}) \notin
\mathbf{V}(g_{1}, g_{2})$ and in that case by the extension theorem $\ref{5a}$ then there exists $u\in \mathbb{C}$
so that $(u,v, S_{5}) \in \mathbf{V}(I_{5})$, so all partial solutions $\mathbf{V}(I_{5u})=((v,S_{5})|v\neq 0,v\neq 1)$ extend as solutions of the original system
$(\ref{222})$. Since the term $C(v,S_{5})$ is of degree $15$ with respect to the variable $v$, it follows by the fundamental theorem of algebra that the equation $C(v,S_{5})=0$ has at most $15$ different roots.
For example, for the value $S_{5}=0$ we obtain the Gr\"obner-basis of the elimination polynomial
\begin{eqnarray*}
G_{5u}&=&27\,{v}^{15}-684\,{v}^{13}+556\,{v}^{12}+4002\,{v}^{11}-4336\,{v}^{10}-8380\,{v}^{9}+14868\,{v}^{8}+ \\ && 4003\,{v}^{7} -19924\,{v}^{6}+5736\,{v}^{5}+
10380\,{v}^{4}-4909\,{v}^{3}-1934\,{v}^{2}+1347\,v-27,
\end{eqnarray*}
for which the variety $\mathbf{V}(I_{5u})$ includes $15$ different values. From these five are real and the rest ten are complex numbers. According to the extension theorem, for every pair of points
$(v_{1},0),\ldots ,(v_{15},0)$ we find the corresponding value of the variable $u$ so that $(u_{1},v_{1},0),\ldots
,(u_{15},v_{15},0) \in \mathbf{V}(I_{5})$. Consequently the sum of period five cycle points attains the same value at most three times.

\end{proof}
 We obtain also the same result if we eliminate the variable $v$ from the pair of equations ($\ref{222}$) using the ordering $\prec_{lex}$, where $S_{5} \prec_{lex} u \prec_{lex} v$.

\section{Conflict of interest}

The author declares that there is no conflict of interest regarding the publication of this paper.


\vspace{1.5cm}

\noindent
\small{\textsc{Pekka \ Kosunen}}\\
\small{Department of Physics and Mathematics, University of Eastern Finland,
P.O.\ Box 111, FI-80101 Joensuu, Finland}\\
\footnotesize{\texttt{pekka.kosunen@uef.fi}}\\

\end{document}